\title{Positive graphs}
\author{Omar Antol\'in Camarena, Endre Cs\'oka, Tam\'as Hubai,\\ G\'abor Lippner, L\'aszl\'o Lov\'asz}
\date{May 2012}

\documentclass[12pt,a4paper]{article}

\usepackage[T1]{fontenc}
\usepackage[utf8]{inputenc}
\usepackage[american]{babel}
\usepackage{hyperref}
\usepackage{amsfonts}
\usepackage{bbm}
\usepackage{amsthm}
\usepackage{amsmath}
\usepackage{amssymb}
\usepackage{enumitem}
\usepackage{tikz}

\newtheorem{Theorem}{Theorem}
\newtheorem{Lemma}[Theorem]{Lemma}
\newtheorem{Corollary}[Theorem]{Corollary}
\newtheorem{Proposition}[Theorem]{Proposition}
\newtheorem{Conjecture}[Theorem]{Conjecture}
\newtheorem{claim}{Claim}

\newcommand{\E}{{\sf E}}
\newcommand{\F}{\mathcal{F}}
\newcommand{\one}{\mathbbm1}
\renewcommand{\Pr}{{\sf P}}
\renewcommand{\P}{\mathcal{P}}

\newcommand{\Ki}{K_n}
\newcommand{\Kpm}{K_n^{\pm 1}}
\newcommand{\ddp}{\,{\rm d}p}
\newcommand{\dmu}{\,{\rm d}\mu}

\long\def\ignore#1{}

\def\FF{\mathcal{F}}
\def\NN{\mathcal{N}}
\def\RR{\mathcal{R}}
\def\Rbb{\mathbb{R}}

\begin{document}

\maketitle

\begin{abstract}
We study ``positive'' graphs that have a nonnegative homomorphism
number into every edge-weighted graph (where the edgeweights may be
negative). We conjecture that all positive graphs can be obtained by
taking two copies of an arbitrary simple graph and gluing them
together along an independent set of nodes. We prove the conjecture
for various classes of graphs including all trees. We prove a number
of properties of positive graphs, including the fact that they have a
homomorphic image which has at least half the original number of
nodes but in which every edge has an even number of pre-images. The
results, combined with a computer program, imply that the conjecture
is true for all graphs up to $9$ nodes.
\end{abstract}

\tableofcontents

\section{Problem description}

For a graph $G$ we are going to denote the set of its vertices by $V(G)$ and the set of its edges by $E(G)$, but may simply write $V$ and $E$ when the it is clear from the context which graph we are talking about.

Let $G$ and $H$ be two simple graphs. A {\it homomorphism} $G\to H$
is a map $V(G)\to V(H)$ that preserves adjacency. We denote by
$\hom(G,H)$ the number of homomorphisms $G\to H$. We extend this
definition to graphs $H$ whose edges are weighted by real numbers
$\beta_{ij} = \beta_{ji}$ ($i,j\in V(H)$):
\[
\hom(G,H)=\sum_{f:~V(G)\to V(H)}\prod_{ij\in E(H)}
\beta_{f(i)f(j)}.
\]
(One could extend it further by allowing nodeweights, and also by
allowing weights in $G$. Positive nodeweights in $H$ would not give
anything new; whether we get anything interesting through weighting
$G$ is not investigated in this paper.)

We call the graph $G$ {\it positive} if $\hom(G,H)\ge0$ for every
edge-weighted graph $H$ (where the edgeweights may be negative). It
would be interesting to characterize these graphs; in this paper we
offer a conjecture and line up supporting evidence.

We call a graph \emph{symmetric}, if its vertices can be partitioned
into three sets $(S, A, B)$ so that $S$ is an independent set, there
is no edge between $A$ and $B$, and there exists an isomorphism
between the subgraphs spanned by $S \cup A$ and $S \cup B$ which
fixes $S$.

\begin{Conjecture} \label{conj}
A graph $G$ is positive if and only if it is symmetric.
\end{Conjecture}

There is an analytic definition for graph positivity which is
sometimes more convenient to work with. A \emph{kernel} is a
symmetric bounded measurable function $W : [0, 1]^2 \rightarrow\Rbb$. A map $p: V(G) \to [0,1]$ can be thought of as a homomorphism into $W$. It also naturally induces a map $p : E(G) \to [0,1]^2$. The
\emph{weight} of $p\in [0,1]^{V(G)}$ is then defined as
\begin{equation*}
\hom(G, W, p) = \prod_{e \in E} W\big(p(e)\big)= \prod_{(a, b) \in E}
W\big(p(a), p(b)\big).
\end{equation*}
The {\it homomorphism density} of a graph $G = (V, E)$ in a kernel
$W$ is defined as the expected weight of a random map:
\begin{equation} \label{tdef}
t(G, W) = \int\limits_{[0, 1]^V} \hom(G, W, p) \ddp = \int\limits_{[0,
1]^V} \prod_{e \in E} W\big(p(e)\big) \ddp.
\end{equation}
Graphs with real edge weights can be considered as kernels in a
natural way: Let $H$ be a looped-simple graph with edge weights
$\beta_{ij}$; assume that $V(H)=[n]=\{1,\dots,n\}$. Split the
interval $[0,1]$ into $n$ intervals $J_1,\dots,J_n$ of equal length,
and define
\[
W_H(x,y)=\beta_{ij}\quad\text{for}\quad x\in J_i,~y\in J_j.
\]
Then it is easy to check that for every simple graph $G$ and
edge-weighted graph $H$, we have $t(G,W_H)=t(G,H)$, where $t(G,H)$ is
a normalized version of homomorphism numbers between finite graphs:
\[
t(G,H)=\frac{\hom(G,H)}{|V(H)|^{|V(G)|}}.
\]
(For two simple graph $G$ and $H$, $t(G,H)$ is the probability that a
random map $V(G)\to V(H)$ is a homomorphism.)

It follows from the theory of graph limits \cite{BCLSV1,LSz1} that
positive graphs can be equivalently be defined by the property that
$t(G,W)\ge0$ for every kernel $W$. 

Hatami \cite{Hat} studied ``norming'' graphs $G$, for which the
functional $W\mapsto t(G,W)^{|E(G)|}$ is a norm on the space of
kernels. Positivity is clearly a necessary condition for this (it is
far from being sufficient, however). We don't know whether our
Conjecture can be proved for norming graphs.

\section{Results}

In this section, we state our results (and prove those with simpler
proofs). First, let us note that the ``if'' part of the conjecture is
easy.

\begin{Lemma}\label{LEM:EASY}
If a graph $G$ is symmetric, then it is positive.
\end{Lemma}

\begin{proof}
For any map $p : V \to [0,1]$ and any subset $M \subset V$ let $p_M$ denote the restriction of $p$ to $M$. Let further $G[M]$ denote the subgraph of $G$ spanned by $M$.
\begin{equation*}
t(G, W) \mathop{=}^{\eqref{tdef}} \int\limits_{[0, 1]^V} \prod_{e \in
E} W\big(p(e)\big)\, \ddp = \int\limits_{[0, 1]^V} \Big(\prod_{e \in G[S
\cup A]} W\big(p(e)\big)\Big) \cdot \Big(\prod_{e \in G[S \cup B]}
W\big(p(e)\big)\Big)\, \ddp
\end{equation*}
\begin{equation*}
= \int\limits_{[0, 1]^S} \Big(\int\limits_{[0, 1]^A} \prod_{e \in G[S
\cup A]} W\big(p(e)\big) \ddp_A\Big) \cdot \Big(\int\limits_{[0, 1]^B}
\prod_{e \in G[S \cup B]} W\big(p(e)\big) \ddp_B\Big) \ddp_S
\end{equation*}
\begin{equation*}
= \int\limits_{[0, 1]^S} \Big(\int\limits_{[0, 1]^A} \prod_{e \in G[S
\cup A]} W\big(p(e)\big) \ddp_A\Big)^2 \ddp_S \ge = 0. \qedhere
\end{equation*}
\end{proof}

In the reverse direction, we only have partial results. We are going
to prove that the conjecture is true for trees (Corollary
\ref{COR:TREES}) and for all graphs up to $9$ nodes (see Section
\ref{SEC:COMP-RES}).

We state and prove a number of properties of positive graphs. Each of
these is of course satisfied by symmetric graphs.

\begin{Lemma}\label{LEM:evenedges}
If $G$ is positive, then $G$ has an even number of edges.
\end{Lemma}

\begin{proof}
Otherwise, choosing $W$ to be the constant $-1$ kernel we get $t(G, W) = (-1)^{|E(G)|} = -1$.
\end{proof}

We call a homomorphism \emph{even} if the preimage of each edge is
has even cardinality.

\begin{Lemma}
If $G$ is positive, then there exists an even homomorphism of $G$
into itself.
\end{Lemma}

\begin{proof}
Let $H$ be obtained from $G$ by assigning random $\pm 1$ weights to its edges, and
let $f$ be a random map $V(G)\to V(H)$. Then
$\E_f(\hom(G,H,f))=t(G,H)\ge 0$, and $t(G,H)>0$ if all weights
are 1, so $\E_H\E_f(\hom(G,H,f))>0$. Hence there is a $f$
for which $\E_H(\hom(G,H,f))>0$. But clearly
$\E_H(\hom(G,H,f))=0$ unless $f$ is an even homomorphism of $G$
into itself.
\end{proof}

Let $K_n$ denote the complete graph on the vertex set $[n]$, where
$n\ge|V(G)|$.

\begin{Theorem} \label{evenhalf}
If a graph $G$ is positive, then there exists an even homomorphism
$f:~G \rightarrow \Ki$ so that $\big|f(V(G))\big| \ge \frac{1}{2}
\big|V(G)\big|$.
\end{Theorem}

We will prove this theorem in Section \ref{proofevenhalf}.

There are certain operations on graphs that preserve symmetry. Every
such operation should also preserve positiveness. We are going to
prove three results of this kind; such results are also useful in
proving the conjecture for small graphs.

We need some basic properties of the homomorphism density function:
Let $G_1$ and $G_2$ be two simple graphs, and let $G_1G_2$ denote
their disjoint union. Then for every kernel $W$
\begin{equation}\label{EQ:PROD}
t(G_1G_2,W)=t(G_1,W)t(G_2,W).
\end{equation}
For two looped-simple graphs $G_1$ and $G_2$, we denote by $G_1\times
G_2$ their {\it categorical product}, defined by
\begin{align*}
V(G_1\times G_2)&=V(G_1)\times V(G_2),\\
E(G_1\times G_2)
&=\bigl\{\bigl((i_1,i_2),(j_1,j_2)\bigr):~(i_1,j_1)\in E(G_1),
(i_2,j_2)\in E(G_2)\bigr\}.
\end{align*}
We note that if at least one of $G_1$ and $G_2$ is simple (has no
loops) then so is the product. The quantity $t(G_1\times G_2,W)$
cannot be expressed as simply as \eqref{EQ:PROD}, but the following
formula will be good enough for us. For a kernel $W$ and
looped-simple graph $H$, let us define the function $W^H:~([0,1]^V)^2
\to \RR$ by
\begin{equation}\label{EQ:WFDEF}
W^H\bigl((x_1,\dots,x_k),(y_1,\dots,y_k)\bigr) = \prod_{(i,j)\in
E(H)} W(x_i,y_j)
\end{equation}
(every non-loop edge of $H$ contributes two factors in this product).
Then we have
\begin{equation}\label{EQ:W-EXP}
t(G\times H,W)=t(G,W^H).
\end{equation}

The following lemma implies that it is enough to prove the conjecture
for connected graphs.

\begin{Lemma}\label{LEM:DISCONN}
A graph $G$ is positive if and only if every connected graph that
occurs among the connected components of $G$ an odd number of times
is positive.
\end{Lemma}

\begin{proof}
The ``if'' part is obvious by \eqref{EQ:PROD}. To prove the converse,
let $G_1,\dots,G_m$ be the connected components of a positive graph
$G$. We may assume that these connected components are different and non-positive, since omitting a positive component or two
isomorphic components does not change the positivity of $G$. We want to
show that $m=0$. Suppose that $m\ge1$.

\begin{claim}\label{CLAIM:1}
We can choose kernels $W_1,\dots,W_m$ so that $t(G_i,W_i)<0$ and
$t(G_i,W_j)\not= t(G_j,W_j)$ for $i\not=j$.
\end{claim}

For every $i$ there is a kernel $W_i$ such that $t(G_i,W_i)<0$, since
$G_i$ is not positive. Next we show that for every $i\not=j$ there is
a kernel $W_{ij}$ such that $t(G_i,W_{ij})\not= t(G_j,W_{ij})$. If
$|V(G_i)|\not=|V(G_j)|$ then the kernel $W_{ij}=\one(x,y\le1/2)$ does
the job, as in this case, due to the connectivity of the graphs, $t(G_i,W_{ij}) = (1/2)^{|V(G_i)|}$. So we may suppose that $|V(G_i)|=|V(G_j)|$. Then by \cite[Theorem 5.29]{lovaszbook} there is a simple
graph $H$ such that $\hom(G_i,H)\not= \hom(G_j,H)$, and hence we can
choose $W_{ij}=W_H$.

Let us denote $\underline{x} = (x_1,\dots,x_m)$ and define $W_j'(\underline{x})=W_j +\sum_{i\not=j} x_i W_{ij}$. Expanding the product in the definition of $t(-,-)$ one easily sees that $Q_j(\underline{x}) = t(G_i,W_j'(\underline{x}))$
$(i=1,\dots,m)$ are all different polynomials in the variables $\underline{x}$, and
hence their values are all different for a generic choice of $\underline{x}$.
If $\underline{x}$ is chosen close to $\underline{0}$, then $t(G_j,W_j'(\underline{x}))<0$, and hence
we can replace $W_j$ by $W_j'(\underline{x})$. This proves the Claim.

Let $W_0= 1$ denote the identically 1 kernel. For nonnegative integers
$k_0,\dots,k_m$, construct a kernel $W_{k_0,\dots,k_m}$ by arranging
$k_i$ rescaled copies of $W_i$ for each $i$ on the ``diagonal". Then
\[
t(G_1\dots G_m, W_{k_0,\dots,k_m}) \mathop{=}^{\eqref{EQ:PROD}}\left(\sum k_i\right)^{-\sum |V(G_j)|}\prod_{j=1}^m \Bigl(\sum_{i=0}^m k_i t(G_j,W_i)\Bigr).
\]
We know that this expression is nonnegative for every choice of the
$k_i$. Since the right hand side is homogeneous in $k_0,\dots,k_m$,
it follows that
\begin{equation}\label{EQ:GW1}
\prod_{j=1}^m \Bigl(1+\sum_{i=1}^m x_i t(G_j,W_i)\Bigr)\ge0
\end{equation}
for every $x_1,\dots,x_m\ge0$. But the $m$ linear forms
$\ell_j(x)=1+\sum_{i=1}^m x_i t(G_j,W_i)$ are different by the choice
of the $W_i$, and each of them vanishes on some point of the positive
orthant since $t(G_j,W_j)<0$. Hence there is a point $x\in\Rbb_+^m$
where the first linear form vanishes but the other forms do not. In a
small neighborhood of this point the product \eqref{EQ:GW1} changes
sign, which is a contradiction.
\end{proof}

\begin{Proposition}\label{PROP:PROD-POS}
If $G$ is a positive simple graph and $H$ is any looped-simple graph,
then $G\times H$ is positive.
\end{Proposition}

\begin{proof}
Immediate from \eqref{EQ:W-EXP}.
\end{proof}

Let $G(r)$ be the graph obtained from $G$ by replacing each node with
$r$ twins of it. Then $G(r)\cong G \times K_r^\circ$, where
$K_r^\circ$ is the complete $r$-graph with a loop added at every
node. Hence we get:

\begin{Corollary}\label{COR:BLOW-POS}
If $G$ is a positive simple graph, then so is $G(r)$ for every positive integer $r$.
\end{Corollary}

As a third result of this kind, we will show that the subgraph of a
positive graph spanned by nodes with a given degree is also positive
(Corollary \ref{COR:DEG}). This proof, however, is more technical and
is given in the next section. Unfortunately, these tools do not help
us much for regular graphs $G$.

\section{Subgraphs of positive graphs}

In this section we develop a technique to show that one can partition the vertices of a positive graph in a certain way so that subgraphs spanned by each part are also positive.  The main idea is to limit, over what maps $p : V \to [0,1]$ one has to average to check positivity. Using this idea recursively we can finally reduce to maps that take each partition to disjoint subsets of $[0,1]$. This in turn allows us to conclude positivity of the spanned subgraphs.

To this end, first we have to introduce the notion of $\mathcal{F}$-positivity. 
Let $G=(V,E)$ be a simple graph. For a measurable
subset $\F\subseteq [0, 1]^V$ and a bounded measurable weight
function $\omega:~[0,1]\to(0,\infty)$, we define
\begin{equation}\label{EQ:tdef2}
t(G, W, \omega, \F) = \int\limits_{p \in \F} \hom(G,W,\omega, p) \ddp, 
\end{equation}
where the \textit{weight} of a $p : V \to [0,1]$ is
\begin{equation}
\hom(G,W,\omega,p) = 
\prod_{v\in V}
\omega\bigl(p(v)\bigr) \prod_{e \in E} W\big(p(e)\big) 
\end{equation}
With the measure $\mu$ with density function $\omega$ (i.e.,
$\mu(X)=\int_X\omega$), we can write this as
\begin{equation}\label{EQ:tdef2a}
t(G, W, \omega, \F) = \int\limits_{\F}\prod_{e \in E} W\big(p(e)\big)
\dmu^V(p).
\end{equation}

We say that $G$ is \emph{$\F$-positive} if for every kernel $W$ and
function $\omega$ as above, we have $t(G, W, \omega, \F) \ge 0$. 
It is easy to see that $G$ is $[0,1]^V$-positive if and only if it is
positive.

We say that $\F_1,\F_2\subseteq[0,1]^V$ are \emph{equivalent} if
there exists a bijection $\varphi:~[0, 1] \rightarrow [0, 1]$ such that
both $\varphi$ and $\varphi^{-1}$ are measurable, and $p \in \F_1 \Leftrightarrow
\varphi(p) \in \F_2$, where $\varphi(p)(v) = \varphi \big(p(v)\big)$.

\begin{Lemma}
If $\F_1$ and $\F_2$ are equivalent, then $G$ is $\F_1$-positive if
and only if it is $\F_2$-positive.
\end{Lemma}

\begin{proof}
Let $\varphi$ denote the bijection in the definition of the equivalence.
For a kernel $W$ and weight function $\omega$, define the kernel
$W^\varphi(x, y) = W\big(\varphi(x), \varphi(y)\big)$, and weight function
$\omega^\varphi(x)=\omega\bigl(\varphi(x)\bigr)$, and let $\mu$ and $\mu_\varphi$
denote the measures defined by $\omega$ and $\omega^\varphi$, respectively.
With this notation,
\begin{align*}
t(G, W^\varphi, \omega^\varphi, \F_2) &= \int\limits_{\F_2} \prod_{e \in E}
W^\varphi \big(p(e)\big) \dmu_\varphi^V(p)\\ &= \int\limits_{\F_1} \prod_{e \in E}
W\big(p(e)\big) \dmu^V(p) = t(G, W, \omega, \F_1).
\end{align*}
This shows that if $G$ is $\F_2$-positive if and only if it is
$\F_1$-positive.
\end{proof}

For a nonnegative kernel $W:~[0, 1]^2 \rightarrow [0, 1]$ (these are
also called {\it graphons}), function $\omega:~[0,1]\to[0,\infty)$,
and $\FF\subseteq[0,1]^V$, define
\begin{equation} \label{sdef}
s = s(G,W,\omega,\F)= \sup\limits_{p \in \F} \Big(\prod\limits_{v \in
V} \omega\big(p(v)\big) \cdot \prod\limits_{e \in E}
W\big(p(e)\big)\Big),
\end{equation}
and
\begin{equation*}
\F_{max} = \Big\{p \in \F:~\prod\limits_{v \in V}
\omega\bigl(p(v)\bigr) \cdot \prod\limits_{e \in E} W\big(p(e)\big) =
s\Big\}.
\end{equation*}
If the Lebesgue measure $\lambda(\F_{max}) > 0$, then we say that
$\F_{max}$ is \emph{emphasizable} from $\F$, and $(W,\omega)$
\textit{emphasizes} it.

\begin{Lemma} \label{pos1}
If $G$ is $\F_1$-positive and $\F_2$ is emphasizable from $\F_1$,
then $G$ is $\F_2$-positive.
\end{Lemma}

\begin{proof}
Suppose that $(U, \tau)$ emphasizes $\F_2$ from $\F_1$, and let $s
= s(G,U,\tau,\F_1)$. Assume that $G$ is not $\F_2$-positive, then
there exists a kernel $W$ and a weight function $\omega$ with $t(G,
W, \omega, \F_2) < 0$. Consider the kernel $W_n=U^nW$ and weight
function $\omega_n=s^{-n/|V|}\tau^n\omega$. Then
\[
\prod_{v \in V} \omega_n\big(p(v)\big) \cdot \prod_{e \in E}
W_n\big(p(e)\big) =\Big(\prod_{v \in V} \omega\big(p(v)\big) \cdot
\prod_{e \in E} W\big(p(e)\big)\Big)\cdot a(p)^n,
\]
where
\[
a(p)= \frac{1}{s} \prod\limits_{v \in V} \tau\big(p(v)\big) \cdot
\prod\limits_{e \in E} U\big(p(e)\big)
  \begin{cases}
   =1  & \text{if $p\in\FF_2$}, \\
   <1  & \text{otherwise}.
  \end{cases}
\]
Thus (by the dominated convergence theorem)
\begin{align*}
t(G, W_n, \omega_n, \F_1) &= \int\limits_{\F_1} \prod_{v \in V}
\omega_n\big(p(v)\big) \cdot \prod_{e \in E} W_n\big(p(e)\big)\, \ddp\\
&\to \int\limits_{\F_2} \prod_{v \in V} \omega\big(p(v)\big) \cdot
\prod_{e \in E} W\big(p(e)\big)\, \ddp = t(G, W, \omega, \F_2) < 0,
\end{align*}
which implies that $G$ is not $\F_1$-positive.
\end{proof}

For a partition $\P$ of $[0, 1]$ into a finite number of sets with
positive measure and a function $\pi:~V \rightarrow \P$, we call the box
$\FF(\pi)=\{p\in [0,1]^V:~p(v)\in \pi(v)~\forall v\in V\}$ a
\emph{partition-box}. Equivalently, a partition-box is a product set
$\prod_{v\in V} S_v$, where the sets $S_v\subseteq[0,1]$ are
measurable, and either $S_u\cap S_v=\emptyset$ or $S_u=S_v$ for all
$u,v\in V$.

A partition $\NN$ of $V$ is \textit{positive} if for any partition $\P$ as above, and any $\pi : V \to \P$ such that $\pi^{-1}(\P) = \NN$, $G$ is $\F(\pi)$-positive.

\begin{Lemma} \label{pos2}
If $\F_1\supseteq\F_2$ are partition-boxes, and $G$ is
$\F_2$-positive, then it is $\F_1$-positive.
\end{Lemma}

\begin{proof}
Let $\F_i$ be a product of classes of partition $\P_i$; we may assume
that $\P_2$ refines $\P_1$. For $P\in\P_2$, let $\overline{P}$ denote
the class of $\P_1$ containing $P$. Since every definition is invariant under measure preserving automorphisms of $[0,1]$, we may assume that every
partition class of $\P_1$ and $\P_2$ is an interval.

Consider any kernel $W$ and any weight function $\omega$. Let
$\varphi:~ [0, 1] \rightarrow [0, 1]$ be the function that maps each
$P\in\P_2$ onto $\overline{P}$ in a monotone increasing and affine way. The map
$\varphi$ is measure-preserving, because for each $A \subseteq Q \in
\P_1$,
\begin{equation} \label{muinv}
\lambda\big(\varphi^{-1}(A)\big) = \sum_{P \in \P_2\atop P\subseteq
Q} \lambda\big(\varphi^{-1}(A) \cap P\big) = \sum_{P \in \P_2\atop P
\subseteq Q} \lambda(A) \frac{\lambda(P)}{\lambda(Q)} = \lambda(A).
\end{equation}
Applying $\varphi$ coordinate-by-coordinate we get a measure
preserving map $\psi:~[0,1]^V\to[0,1]^V$. Then $\psi'=\psi|_{\F_2}$
is an affine bijection from $\F_2$ onto $\F_1$, and clearly
$\det(\psi')>0$. Hence
\begin{align*}
t(G, W^\varphi, \omega^\varphi, \F_2) &\mathop{=}^{\eqref{EQ:tdef2}}
\int\limits_{\F_2} \prod_{v \in V} \omega^\varphi\big(p(v)\big) \cdot
\prod_{e \in E} W^\varphi\big(p(e)\big) \ddp\\
&= \int\limits_{\F_2} \prod_{v \in V} \omega \big((\psi'(p))(v)\big) \cdot
\prod_{e \in E} W\big((\psi'(p))(e)\big) \ddp\\
&= \det(\psi')^{-1}  \cdot\int\limits_{\F_1} \prod_{v \in V}
\omega\big(p(v)\big)\cdot \prod_{e \in E}
W\big(p(e)\big) \ddp\\
&\mathop{=}^{\eqref{EQ:tdef2}} \det(\psi')^{-1}\cdot t(G, W, \omega, \F_1).
\end{align*}
Since $G$ is $\F_2$-positive, the left hand side is positive, and
hence $t(G, W, \omega, \F_1)\ge0$, proving that $G$ is
$\F_1$-positive.
\end{proof}

\begin{Corollary}\label{COR:pos2}
If $\NN_2$ refines $\NN_1$ and $\NN_2$ is positive, then $\NN_1$ is positive as well.
\end{Corollary}

\begin{Lemma} \label{constr}
Suppose that $\F_1$ is a partition-box defined by a partition $\P$
and function $\pi_1$. Let $Q\in\P$ and let $U$ be the union of an
arbitrary set of classes of $\P$. Let $\theta$ be a positive number
but not an integer. Split $Q$ into two parts with positive measure,
$Q^+$ and $Q^-$. Let $\deg(v,U)$ denote the number of neighbors $u$
of $v$ with $\pi_1(u)\subseteq U$. Define
\begin{alignat*}{3}
\pi_2(v) &=
\begin{cases}
   \pi_1(v) &\text{ if } \pi_1(v) \ne Q,
\\ Q^+  &\text{ if } \pi_1(v) = Q \text{ and } \deg(v, U) > \theta,
\\ Q^-  &\text{ if } \pi_1(v) = Q \text{ and } \deg(v, U) < \theta,
\end{cases}
\end{alignat*}
and let $\F_2$ be the corresponding partition-box. Then there exists
a pair $(W, \omega)$ emphasizing $\F_2$ from $\F_1$.
\end{Lemma}

\begin{proof}
Clearly, $\lambda(\F_2) > 0$. First, suppose that $Q\not\subseteq U$.
Let $W$ be 2 in $Q^+ \times U$ and in $U \times Q^+$, and 1
everywhere else. Let $\omega(x)$ be $2^{-\theta}$ if $x \in Q^+$ and 1
otherwise. It is easy to see that the weight of a $p \in \F_1$ is
$2^a$, where $a=\sum_{v \in p^{-1}(Q^+)} \big(\deg(v, U) - \theta \big)$.
This expression is maximal if and only if $p \in \F_2$. 

In the case when
$Q \subset U$ the only difference is that one has to let $W = 4$ in the intersection $Q^+ \times U \cap U \times Q^{+}$. 
\end{proof}

\begin{Corollary}\label{COR:constr}
If $\NN_1$ is a positive partition of the vertex set, $U$ is an arbitrary union of classes, $Q$ is a single class, $\theta > 0$ is not an integer, and $\NN_2$ is obtained from $\NN_1$ by splitting $Q$ according to whether (by abuse of notation) $\deg(v,U) > \theta$ or not for each vertex $v \in Q$, then $\NN_2$ is also positive.
\end{Corollary}

We can use Corollary \ref{COR:constr} iteratively: we start with the
trivial partition, and refine it so that it remains positive.
This is essentially the 1-dimensional Weisfeiler--Lehman algorithm, which classifies vertices recursively, see e.~g.~\cite{Douglas} It starts splitting vertices into classes according to their degree. Then in each step it refines the existing classes according to the number of neighbors in each of the current classes. The analogy will be clear from the proofs below. 
There is a
non-iterative description of the resulting partition, and this is
what we are going to describe next.

The \emph{walk-tree} of a rooted graph $(G, v)$ is the following
infinite rooted tree $R(G, v)$: its nodes are all finite walks
starting from $v$, its root is the 0-length walk, and the parent of
any other walk is obtained by deleting its last node. The \emph{walk-tree partition} $\RR$ is
the partition of $V$ in which two nodes $u,v\in V$ belong to the same
class if and only if $R(G,u)\cong R(G,v)$.

\begin{Proposition} \label{prop:walk}
If a graph $G$ is positive, then its walk-tree partition is also positive.
\end{Proposition}

\begin{proof}
Let the $k$-neighborhood of $r$ in $R(G, r)$ be denoted by $R_k(G,
r)$.  The \textit{$k$-walk-tree partition} $\RR_k$ is the partition of $V$ in which two nodes $u,v \in V$ belong to the same class if and only if $R_k(G,u) \cong R_k(G,v)$.
 Clearly, if for two vertices $R(G,u) \neq R(G,v)$ then there is a $k = k(u,v)$ such that $R_k(G,u) \neq R_k(G,v)$. Since $V$ is finite, choosing $k_0 = \max_{u,v \in V} k(u,v)$ we see that $\RR_{k_0} = \RR$. Thus we are done if we show that $\RR_k$ is positive for every $k$.

We prove this by induction. If $k = 0$ then $\RR_0$ is the trivial partition, hence the assertion follows from the positivity of $G$.  Now let us assume that the statement is true
for $k$. Clearly, $R_{k+1}(G,v)$ is determined by the \textit{neighborhood profile}, the multi-set $\{ R_k(G,u) : u \sim v\}$. Using Corollary~\ref{COR:constr}, we separate each
class $Q$ into subclasses so that $u, v \in Q$ end up in the same class if and only if  their neighborhood profiles are the same. The new partition is exactly $\RR_{k+1}$. 
\end{proof}

\begin{Corollary} \label{subgraph} Let $G(V,E)$ be a positive graph, and 
let $S \subset V$ be the union of some classes of the walk-tree partition. Then $G[S]$ is also positive.
\end{Corollary}

\begin{proof}
By Corollary~\ref{COR:pos2} the partition $\NN = \{ S, V \setminus S\}$ is positive. Let $\P = \{ [0,1/2] , (1/2, 1] \}$ and define $\pi : V \to \P$ by $\pi(v) = [0,1/2]$ if and only if $v \in S$.  Suppose that $G[S]$ is negative as demonstrated by some $W$. Let us define
\[ W'(x,y) = \left\{ \begin{array}{ll} W(2x,2y) & : x,y \in [0,1/2] \\ 1 & : \mbox{otherwise} \end{array} \right.\]
Then $t(G,W',1,\F(\pi)) < 0$ contradicting the positivity of the partition $\NN$.
\end{proof}

\begin{Corollary}\label{COR:DEG}
If $G$ is positive, then for each $k$ the subgraph of $G$ spanned by
all nodes with degree $k$ is positive as well. \qed
\end{Corollary}

\begin{Corollary}
For each odd $k$ the number of nodes of $G$ with degree $k$ must be
even.
\end{Corollary}

\begin{proof}
Otherwise, consider the partition-box $\F$ that separates the
vertices of $G$ with degree $d$ to class $A=[0, 1/2]$ and the other
vertices to $\bar{A}=(1/2,1]$. Consider the kernel $W$ which is $-1$
between $A$ and $\bar{A}$ and 1 in the other two cells. Then for each
map $p\in[0,1]^V$, the total degree of the nodes mapped into class
$A$ is odd, so there is an odd number of edges between $A$ and
$\bar{A}$. So the weight of $p$ is $-1$, therefore $t(G, W, 1, \F) =
-\lambda(\F) < 0$.
\end{proof}

\begin{Corollary}\label{COR:TREES}
Conjecture \ref{conj} is true for trees.
\end{Corollary}

\begin{proof}
From the walk-tree of a vertex $v$ of the tree $G$, we can easily
decode the rooted tree $(G,v)$. We call a vertex {\it central} if
it cuts $G$ into components with at most $|V|/2$ nodes. There can be
either one central node or two neighboring central nodes of $G$. If
there are two of them, then their walk-trees are different from the
walk-trees of every other node. But these two points span a graph with a single edge,
which is not positive, therefore Corollary \ref{subgraph} implies that
neither is $G$. If there is only one central node, then consider the
walk-trees of its neighbors. If there is an even number of each kind,
then $G$ is symmetric (and is thus positive by Lemma~\ref{LEM:EASY}). Otherwise we can find two classes (one consist of the central node, the other consists of an odd number of its neighbors) whose union spans a graph with an odd number of edges, hence it is negative by Lemma~\ref{LEM:evenedges}.
\end{proof}

\section{Homomorphic images of positive graphs} \label{proofevenhalf}

The main goal of this section is to prove Theorem \ref{evenhalf}. In
what follows, let $n$ be an integer. For a homomorphism $f:~ G
\rightarrow \Ki$, we call an edge $e \in E(\Ki)$ \emph{$f$-odd} if
$\big| f^{-1}(e) \big|$ is odd. We call a vertex $v \in V(K_n)$
\emph{$f$-odd} if there exists an $f$-odd edge incident with $v$. Let
$E_\text{odd}(f)$ and $V_\text{odd}(f)$ denote the set of $f$-odd
edges and nodes of $\Ki$, respectively, and define
\begin{equation} \label{rdef}
r(f) = \big| V(G) \big| - \big| f\big(V(G)\big) \big| + \frac{1}{2}
|V_\text{odd}(f)|.
\end{equation}

\begin{Lemma}
Let $G_i=(V_i,E_i)$ $(i=1,2)$ be two graphs, let $f:~G_1G_2
\rightarrow \Ki$, and let $f_i:~ G_i \rightarrow \Ki$ denote the
restriction of $f$ to $V_i$. Then $r(f) \ge r(f_1)+r(f_2)$.
\end{Lemma}

\begin{proof}
Clearly $|V(G)|=|V_1|+|V_2|$ and $|V(f(G))|=
|f(V_1)|+|f(V_2)|-|f(V_1)\cap f(V_2)|$. Furthermore,
$E_\text{odd}(f)=E_\text{odd}(f_1)\triangle E_\text{odd}(f_2)$, which
implies that $V_\text{odd}(f)\supseteq V_\text{odd}(f_1)\triangle
V_\text{odd}(f_2)$. Hence
\begin{align*}
|V_\text{odd}(f)|&\ge |V_\text{odd}(f_1)|+ |V_\text{odd}(f_2)| - 2
|V_\text{odd}(f_1)\cap V_\text{odd}(f_2)|\\
&\ge |V_\text{odd}(f_1)|+ |V_\text{odd}(f_2)| - 2 |f(V_1)\cap
f(V_2)|.
\end{align*}
Substituting these expressions in \eqref{rdef}, the lemma follows.
\end{proof}

Let $G^k$ denote the disjoint union of $k$ copies of a graph $G$.
This lemma implies that if $f:~G^k \rightarrow \Ki$ is any
homomorphism and $f_i:~ G \rightarrow \Ki$ denotes the restriction of
$f$ to the $i$-th copy of $G$, then
\begin{equation} \label{rsupadd}
r(f) \ge \sum\limits_{i=1}^{k} r(f_i).
\end{equation}

We define two parameters of a graph $G$:
\begin{equation} \label{rbardef}
\bar{r}(G) = \min \big\{ r(f) \big| n \in \mathbb{N}, f:~G \rightarrow \Ki \big\}
\end{equation}
and 
\begin{equation} \label{pdef}
q(G) = \min \Big\{|V(G)| - |f(V(G))|~\big|~ n \in \mathbb{N}, f:~G \rightarrow
\Ki\text{ is even} \Big\}.
\end{equation}
If there is no even homomorphism from $G$ to $K_n$ for any $n$ then we define $q(G) = \infty$.
Since $q(G) =\min \big\{ r(f) \big| n \in \mathbb{N}, f:~G \rightarrow \Ki\text{ is
even}\big\}$, it follows that
\begin{equation}\label{EQ:RBAR-P}
q(G) \ge \bar{r}(G).
\end{equation}
Furthermore, considering any injective $f:~G\to\Ki$, we see that
\begin{equation}\label{EQ:INJ}
\bar{r}(G) \le r(f) = \big|V(G) \big| - \big| f\big(V(G)\big) \big| +
\frac{1}{2} \big| f\big(V(G)\big) \big| = \frac12 \big|V(G) \big|.
\end{equation}

\begin{Lemma}
\begin{equation} \label{rbarmulti}
\bar{r}(G^k) = k \bar{r}(G).
\end{equation}
\end{Lemma}

\begin{proof}
For one direction, take an $f:~ G^k\rightarrow \Ki$ with $r(f) =
\bar{r}(G^k)$. Then
\begin{equation*}
\bar{r}(G^k) = r(f) \mathop{\ge}^{\eqref{rsupadd}} \sum_{i = 1}^k
r(f_i) \mathop{\ge}^{\eqref{rbardef}} \sum_{i = 1}^k \bar{r}(G) = k
\cdot \bar{r}(G).
\end{equation*}
For the other direction, let us choose each $f_i$ so that $r(f_i) =
\bar{r}(G)$ and the images $f_i(G)$ are pairwise disjoint. Then
\begin{equation*}
\bar{r}(G^k) \mathop{\le}^{\eqref{rbardef}} r(f) = \sum_{i = 1}^k
r(f_i) = \sum_{i = 1}^k \bar{r}(G) = k \cdot \bar{r}(G). \qedhere
\end{equation*}
\end{proof}

\begin{Lemma}
\begin{equation} \label{p2G}
q(G^2) = \bar{r}(G^2).
\end{equation}
\end{Lemma}

\begin{proof}
We already know by \eqref{EQ:RBAR-P} that $q(G^2) \ge \bar{r}(G^2)$.
For the other direction, we define $f:~ G^2\rightarrow \Ki$ as
follows. We choose $f_1$ so that $r(f_1) = \bar{r}(G)$. Consider all
points $v_1, v_2, ..., v_l$ in $f_1\big(V(G)\big)$ which are not
$f_1$-odd. Let us choose pairwise different nodes $v'_1, v'_2, ...,
v'_l$ disjointly from $f_1\big(V(G)\big)$. Now we choose $f_2$ so that
for each $x \in V(G)$, if $f_1(x)$ is an $f_1$-odd point, then
$f_2(x) = f_1(x)$, and if $f_1(x) = v_i$, then $f_2(i) = v'_i$.

If an edge $e \in E(\Ki)$ is incident to a $v_i$, then $\big|
f_1^{-1}(e) \big|$ is even and $f_2^{-1}(e) = \emptyset$. If $e$ is
incident to a $v'_i$, then $\big| f_2^{-1}(e) \big|$ is even and
$f_1^{-1}(e) = \emptyset$. If $e$ is not incident to any $v_i$ or
$v'_i$, then $\big| f_1^{-1}(e) \big| = \big| f_2^{-1}(e) \big|$.
Therefore $f$ is even. Thus,
\begin{equation*}
q(G^2) \mathop{\le}^{\eqref{pdef}} r(f) \mathop{=}^{\eqref{rdef}}
\big|V(G^2)\big| - \big|f(V(G^2))\big|
\end{equation*}
\begin{equation*}
= 2 \big|V(G)\big| - \Big|f_1\big(V(G)\big)\Big| -
\Big|f_2\big(V(G)\big)\Big| + \Big|f_1\big(V(G)\big) \cap
f_2\big(V(G)\big)\Big|
\end{equation*}
\begin{equation*}
= 2 \big|V(G)\big| - 2 \Big|f_1\big(V(G)\big)\Big| +
|V_{\text{odd}}(f_1)| \mathop{=}^{\eqref{rdef}} 2 r(f_1) = 2
\bar{r}(G) \mathop{=}^{\eqref{rbarmulti}} \bar{r}(G^2). \qedhere
\end{equation*}
\end{proof}

Let $K_n^w$ denote $K_n$ equipped with an edge-weighting $w:~ E(K_n)
\rightarrow \{-1, 1\}$. Let the stochastic variable $K_n^{\pm 1}$
denote $K_n^w$ with a uniform random $w$.

\begin{Lemma}
For a fixed graph $G$, and $n \rightarrow \infty$,
\begin{equation*}
\E\big(t(G, \Kpm)\big) = \left\{ \begin{array}{ll} \Theta\big(n^{-q(G)}\big) & \text{ if $q(G) < \infty$} \\ 0 & \text{ otherwise.} \end{array} \right.
\end{equation*}
\end{Lemma}

\begin{proof}
If an edge $e$ is $f$-odd, then changing the weight on $e$ changes
the sign of the homomorphism, therefore $\E_w\big(\hom(G, K_n^w,
f)\big) = 0$. On the other hand, if $f$ is even, then for all $w$,
$\hom(G, K_n^w, f) = 1$. Therefore, taking a uniformly random
homomorphism $f:~ G \rightarrow K_n$,
\begin{align*}
\E\big(t(G, \Kpm)\big) &= \E_w\Big(\E_f\big(\hom(G, K_n^w,
f)\big)\Big) = \E_f\Big(\E_w\big(\hom(G, K_n^w, f)\big)\Big)\\
&= \Pr(f\text{ is even}).
\end{align*}
If $q(G) = \infty$ we are done. Otherwise we have
\begin{equation*}
\Pr(f\text{ is even}) \le \Pr\Big( \big|V(G)\big| - \big| f(V(G))
\big| \ge q(G) \Big) = O(n^{-q(G)}).
\end{equation*}
On the other hand, consider an even homomorphism $g:~ G \rightarrow
K_n$ with $r(g) = q(G)$. For each subset $H \subset V(K_n)$ of size $|H| = |g(V(G))|$ there is a permutation $\sigma_H$ on $V(K_n)$ that maps $g(V(G))$ bijectively to $H$. Then $f_H = \sigma(g(x))$ is also an even homomorphism, and clearly $f_{H_1} \neq f_{H_2}$ unless $H_1 = H_2$. Thus there are at least ${n
\choose |g(V(G)|}$ different even homomorphisms $f:~ G \to K_n$.
Therefore
\begin{align*}
\Pr(f\text{ is even}) &\ge \Pr(f = f_H \text{ for some $H$}) =
\binom{n}{|g(V(G))|} \Big/ n^{|V(G)|}\\
&= \Omega(n^{-q(G)}).
\end{align*}
\end{proof}

Now let us turn to the proof of Theorem \ref{evenhalf}. Assume that
$G$ is positive, then the random variable $X = t(G, \Kpm)$ is
nonnegative. Applying the Cauchy-Schwartz inequality to $X^{1/2}$ and
$X^{3/2}$ we get that
\begin{equation} \label{Holder}
\E(X) \cdot \E(X^3) \ge \E(X^2)^2.
\end{equation}
Here
\begin{equation*}
\E(X^k) = \E\big( t(G, \Kpm)^k \big) \mathop{=}^{\eqref{EQ:PROD}}
\E\big( t(G^k, \Kpm) \big) = \Theta\big( n^{-q(G^k)} \big),
\end{equation*}
so \eqref{Holder} shows that $n^{-q(G)} \cdot n^{-q(G^3)} =
\Omega\big(n^{-2 q(G^2)}\big)$, thus $q(G) + q(G^3) \le 2 q(G^2)$.
Hence
\begin{equation} \label{equal}
4 \bar{r}(G) \mathop{=}^{\eqref{rbarmulti}} \bar{r}(G) + \bar{r}(G^3)
\mathop{\le}^{\eqref{EQ:RBAR-P}} q(G) + q(G^3) \le 2 q(G^2) \mathop{=}^{\eqref{p2G}} 2
\bar{r}(G^2) \mathop{=}^{\eqref{rbarmulti}} 4 \bar{r}(G).
\end{equation}
All expressions in \eqref{equal} must be equal, therefore $\bar{r}(G)
= q(G)$.

Finally, for an even $f:~ G \rightarrow \Ki$ with $\big|V(G)\big| -
\big|f(V(G))\big| = q(G)$, we have
\begin{equation*}
\frac{1}{2} \big|V(G)\big| \mathop{\ge}^{\eqref{EQ:INJ}}\bar{r}(G)=
q(G) = \big|V(G)\big| - \big|f(V(G))\big|,
\end{equation*}
therefore $\big|f\big(V(G)\big)\big| \ge \frac{1}{2}\big|V(G)\big|$. \qed

\section{Computational results}\label{SEC:COMP-RES}

We checked Conjecture \ref{conj} for all graphs on at most 9 vertices
using the previous results and a computer program. Starting from the
list of nonisomorphic graphs, we filtered out those who violated one
of our conditions for being a minimal counterexample. In particular
we performed the following tests:

\begin{enumerate}
\item Check whether the graph is symmetric, by
exhaustive search enumerating all possible involutions of the
vertices. If the graph is symmetric, it is not a counterexample.

\item Calculate the number of homomorphisms into graphs represented
by $1\times 1$, $2\times 2$ or $3\times 3$ matrices of small
integers. (Checking $1\times 1$ matrices is just the same as checking
whether or not the number of edges is even.) If we get a negative
homomorphism count, the graph is negative and therefore it is not a
counterexample.

\item Calculate the number of homomorphisms into graphs represented
by symbolic $3\times 3$ and $4\times 4$ matrices and perform local
minimization on the resulting polynomial from randomly chosen points.
Once we reach a negative value, we can conclude that the graph is negative.

\item Partition the vertices of the graph in such a way that two
vertices belong to the same class if and only if they produce the
same walk-tree (1-dimensional Weisfeiler--Lehman algorithm). Check
for all proper subsets of the set of classes whether their union
spans an asymmetric subgraph. If we find such a subgraph, the graph
is not a minimal counterexample: either the subgraph is not positive
and by Corollary \ref{subgraph} the original graph is not positive
either, or the subgraph is positive, and therefore we have a smaller
counterexample.

\item Consider only those homomorphisms which map all vertices in
the $i$th class of the partition into vertices $3i+1$, $3i+2$ and
$3i+3$ of the target graph represented by a symbolic matrix. If we
get a negative homomorphism count, the graph is negative by
Proposition \ref{prop:walk}. (In this case we work with a $3k\times
3k$ matrix where $k$ denotes the number of classes of the walk-tree
partition, but the resulting polynomial still has a manageable size
because we only count a small subset of homomorphisms. Note that if
one of the classes consists of a single vertex, we only need one
corresponding vertex in the target graph.)
\end{enumerate}

The tests were performed in such an order that the faster and more
efficient ones were run first, restricting the later ones to the set
of remaining graphs. For example, in step 4, we start with checking
whether any of the classes spans an odd number of edges, or whether
the number of edges between any two classes is odd. We used the
{\tt SAGE} computer-algebra system for our calculations and
rewritten the speed-critical parts in {\tt C} using {\tt nauty}
for isomorphism checking, {\tt mpfi} for interval arithmetics
and Jean-S\'ebastien Roy's {\tt tnc} package for nonlinear optimization.

Our automated tests left only one graph on 9 vertices as a possible
minimal counterexample, the graph on left:

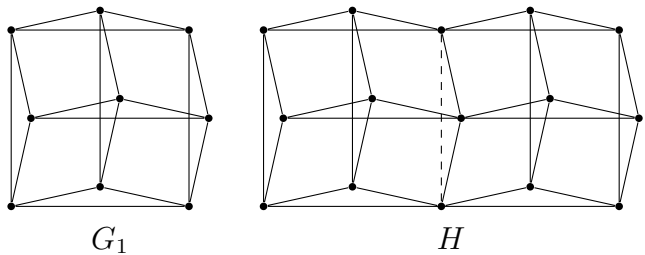
\tikzset{vertex/.style={circle,fill,inner sep=0mm,minimum size=1mm}}
\begin{center}
\begin{tabular}{@{\extracolsep{0.2cm}}cc}
\begin{tikzpicture}[scale=0.75*sqrt(3)]
\foreach \i/\k in {A/0,B/1,C/2} \foreach \j in {0,1,2}
  \path (0.9*\j,0.9*\k) ++(\j*180+270:0.1) ++(\k*180+180:0.1) node[vertex] (\i\j) {};
\foreach \i in {A,B,C} \draw (\i0) -- (\i1) -- (\i2) -- (\i0);
\foreach \i in {0,1,2} \draw (A\i) -- (B\i) -- (C\i) -- (A\i);
circle(0.05cm);
\end{tikzpicture}&
\begin{tikzpicture}[scale=0.75*sqrt(3)]
\foreach \i/\k in {A/0,B/1,C/2} \foreach \j in {0,1,2,3,4}
  \path (0.9*\j,0.9*\k) ++(\j*180+270:0.1) ++(\k*180+180:0.1) node[vertex] (\i\j) {};
\foreach \i in {A,B,C} \draw (\i0) -- (\i1) -- (\i2) -- (\i3) --
(\i4) -- (\i0); \foreach \i in {0,1,3,4} \draw (A\i) -- (B\i) --
(C\i) -- (A\i); \draw (A2) -- (B2) -- (C2); \draw[dashed] (C2) --
(A2); circle(0.05cm);
\end{tikzpicture}\\
$G_1$&$H$\\
\end{tabular}
\end{center}

The non-positivity of this graph was checked manually by counting the
number of homomorphisms into the graph on the right (where the dashed
edge has weight $-1$ and all other edges have weight $1$). This
leaves only the following three of the 12\,293\,435 graphs on at most
10 vertices as candidates for a minimal counterexample:

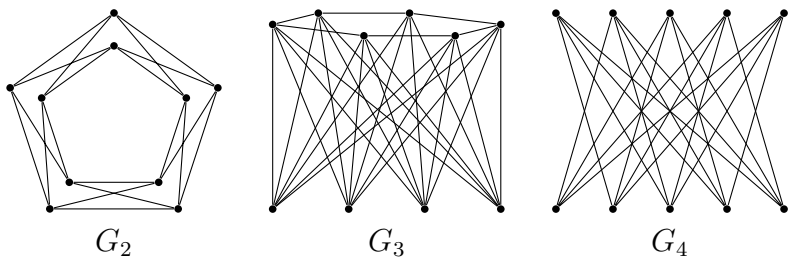
\begin{center}
\begin{tabular}{@{\extracolsep{0.2cm}}ccc}
\begin{tikzpicture}
\pgfmathsetmacro{\r}{sqrt(3-sqrt(5))*sqrt(27/10)} \foreach \i in
{0,...,4} \node[vertex] (A\i) at (90+72*\i:\r) {}; \foreach \i in
{0,...,4} \node[vertex] (B\i) at (90+72*\i:1.0cm) {}; \foreach \i in
{0,...,4} {
  \pgfmathtruncatemacro{\j}{mod(\i+1,5)}
  \draw (A\i) -- (A\j) (A\i) -- (B\j) (B\i) -- (A\j) (B\i) -- (B\j);
}
\end{tikzpicture}&
\begin{tikzpicture}
\foreach \i/\j in {0/1,1/0,2/2,3/0,4/2,5/1} \node[vertex] (A\i) at (0.6*\i,{1.5*sqrt(3)-0.15*\j}) {};
\foreach \i in {0,...,3} \node[vertex] (B\i) at (\i,0) {};
\draw (A0) -- (A1) -- (A3) -- (A5) -- (A4) -- (A2) -- (A0);
\foreach \i in {0,...,5} \foreach \j in {0,...,3} \draw (A\i) -- (B\j);
\end{tikzpicture}&
\begin{tikzpicture}
\foreach \i in {0,...,4} \node[vertex] (A\i) at (0.75*\i,0) {};
\foreach \i in {0,...,4} \node[vertex] (B\i) at (0.75*\i,{1.5*sqrt(3)}) {};
\draw (A0) -- (B1) (A0) -- (B2) (A0) -- (B3) (A0) -- (B4);
\draw (A1) -- (B0) (A1) -- (B2) (A1) -- (B3) (A1) -- (B4);
\draw (A2) -- (B0) (A2) -- (B1) (A2) -- (B3) (A2) -- (B4);
\draw (A3) -- (B0) (A3) -- (B1) (A3) -- (B2) (A3) -- (B4);
\draw (A4) -- (B0) (A4) -- (B1) (A4) -- (B2) (A4) -- (B3);
\end{tikzpicture}\\
$G_2$&$G_3$&$G_4$\\
\end{tabular}
\end{center}

Note that all three graphs are regular, as is the case for all remaining
graphs on 11 vertices. We have found step 5 of the algorithm quite
effective at excluding graphs with nontrivial walk-tree partitions.

\medskip

{\bf Acknowledgement.} The conjecture in this paper was the subject
of a research group at the American Institute of Mathematics workshop
``Graph and hypergraph limits", Palo Alto, CA,  August 15--19, 2011.
We are grateful for the inspiration from all those who took part in
the discussions of this research group, in particular to Sergey Norin
and Oleg Pikhurko.

Further research on the topic of this paper was supported by ERC
Grant No.~227701 and NSF under agreement No. DMS-0835373. Any
opinions and conclusions expressed in this material are those of the
authors and do not necessarily reflect the views of the NSF or of the
ERC.

\end{document}